\documentclass[a4paper,12pt]{article}
\usepackage{amssymb}
\usepackage{amsmath}
\usepackage{amsfonts}
\usepackage{amsthm}
\usepackage{fancyhdr}

\newcommand{\C}{\mathbb{C}}
\newcommand{\Z}{\mathbb{Z}}

\newcommand{\N}{\mathbb{N}}

\newcommand{\G}{\mathcal{G}}
\newcommand{\E}{\mathcal{E}}

\newcommand{\Odd}{\mathcal{O}}

\newcommand\shorttitle{Approximation to Conjecture}
\newcommand\authors{Ronald Orozco L\'opez}

\newtheorem{theorem}{Theorem}

\newtheorem{conjecture}{Conjecture}

\newtheorem{lemma}{Lemma}

\numberwithin{equation}{section}

\begin{document}
\title{\textbf{An Approximation to Proof of the Circulant Hadamard Conjecture}}

\author{Ronald Orozco L\'opez}

\newcommand{\Addresses}{{
  \bigskip
  \footnotesize

  \textsc{Department of Mathematics, Universidad de los Andes,
    Bogot\'a Colombia,}\par\nopagebreak
  \textit{E-mail address}, \texttt{rj.orozco@uniandes.edu.co}

}}

\maketitle

\begin{abstract}
It is known that if a circulant Hadamard matrix of order $n$ exists then $n$ must be of the form $n=4m^{2}$ 
for some odd integer $m$. In this paper we use the structure constant of Schur ring of $\Z_{2}^{4m^{2}}$ 
to prove that there is no circulant Hadamard matrix in $\Z_{2}^{4m^{2}}$ except possibly for sequences with Hamming weight $a+b$, with $\frac{m^{2}-m}{2}\leq a\leq\frac{3m^{2}-m}{2}$ and $b=2m^{2}-m-a$ and with
$\frac{m^{2}+m}{2}\leq 2m^{2}-a\leq\frac{3m^{2}+m}{2}$ and $b=m+a$.
\end{abstract}

\section{Introduction}

An Hadamard matrix $H$ is an $n$ by $n$ matrix all of whose entries are $+1$ or 
$-1$ which satisfies $HH^{t}=nI_{n}$, where $H^{t}$ is the transpose of $H$
and $I_{n}$ is the unit matrix of order $n$. It is also known that, if an
Hadamard matrix of order $n>1$ exists, $n$ must have the value $2$ or be
divisible by 4. There are several conjectures associated with Hadamard matrices. 
The main conjecture concerns its existence. This states that an Hadamard matrix exists for
all multiple order of 4. Another very important conjecture states that no circulant Hadamard matrix
exists if the order is different from 4 [3],[4],[5],[6],[7],[8]. 

On the circulant Hadamard conjecture the first significant result was made by R.J. Turyn [3] using arguments from
algebraic number theory. He prove that if a circulant Hadamard matrix of order $n$ exists then 
$n$ must be of the form $n=4m^{2}$ for some odd integer $m$ which is not a prime-power. Many important 
results about this conjecture can be found in [4],[5],[6],[7],[8]. In this paper we use the structure constant of Schur ring of $\Z_{2}^{4m^{2}}$ to prove that there is no circulant Hadamard matrix in $\Z_{2}^{4m^{2}}$ except
for sequences with Hamming weight $a+b$, with $\frac{m^{2}-m}{2}\leq a\leq\frac{3m^{2}-m}{2}$ and $b=2m^{2}-m-a$
and with $\frac{m^{2}+m}{2}\leq 2m^{2}-a\leq\frac{3m^{2}+m}{2}$ and $b=m+a$.

\section{Schur Ring over $\mathbb{Z}_{2}^{n}$}
Let $G$ be a finite group with identity element $e$ and $\C[G]$ the group algebra of all formal sums 
$\sum_{g\in G}a_{g}g$, $a_{g}\in \C$, $g\in G$. For $T\subset G$, the element $\sum_{g\in T}g$ will 
be denoted by $\overline{T}$. Such an element is also called a $\textit{simple quantity}$. 
The transpose of $\alpha = \sum_{g\in G}a_{g}g$ is defined as $\alpha^{\top} = \sum_{g\in G}a_{g}(g^{-1})$. 
Let $\{T_{0},T_{1},...,T_{r}\}$ be a partition of $G$ and let $S$ be the subspace of $\C[G]$ spanned by 
$\overline{T_{1}},\overline{T_{2}},...,\overline{T_{r}}$.  We say that $S$ is a $\textit{Schur ring}$ 
($S$-ring, for short) over $G$ if: 
\begin{enumerate}
\item $T_{0} = \lbrace e\rbrace$, 
\item for each $i$, there is a $j$ such that $T_{i}^{\top} = T_{j}$,
\item for each $i$ and $j$, we have $\overline{T_{i}}\overline{T_{j}} = \sum_{k=1}^{r}\lambda_{i,j,k}\overline{T_{k}}$, for constants $\lambda_{i,j,k}\in\C$.
\end{enumerate}

The numbers $\lambda_{i,j,k}$ are the structure constants of $S$ with respect to the linear base 
$\{\overline{T_{0}},\overline{T_{1}},...,\overline{T_{r}}\}$. 
The sets $T_{i}$ are called the \textit{basic sets} of the $S$-ring $S$. Any union of them is called an 
$S$-set. Thus, $X\subseteq G$ is an $S$-set if and only if $\overline{X}\in S$. The set of all $S$-set 
is closed with respect to taking inverse and product. Any subgroup of $G$ that is an $S$-set, is called 
an $S$-\textit{subgroup} of $G$ or $S$-\textit{group}. (For details, see [1],[2])\\

In this paper denote by $\Z_{2}$ the cyclic group of order 2 with elements $+$ and $-$(where + and $-$ mean 1 and $-1$ respectively). Let $\Z_{2}^{n}=\overset{n}{\overbrace{\Z_{2}\times \cdots \times \Z_{2}}}$. Then all $X\in\Z_{2}^{n}$ are sequences of $+$ and $-$ and will be called $\Z_{2}-$\textit{sequences}. 

Let $\omega(X)$ denote the Hamming weight of $X\in\Z_{2}^{n}$. Thus, $\omega(X)$ is the number of $+$ in any 
$\Z_{2}-$sequences $X$ of $\Z_{2}^{n}$. Now let $\G_{n}(k)$ be the subset of $\Z_{2}^{n}$ such that $\omega(X)=k$ for all $X\in\G_{n}(k)$, where $0\leq k\leq n$. 

We let $T_{i}=\G_{n}(n-i)$. It is straightforward to prove that the partition $S=\{\G_{n}(0),...,\G_{n}(n)\}$ 
is a partition of $\Z_{2}^{n}$. And also $S$ is an $S-$ring over 
$\Z_{2}^{n}$. From [2] it is know that the constant structure $\lambda_{i,j,k}$ is equal to
\begin{equation}\label{estruc_cons}
\lambda_{i,j,k}=
\begin{cases}
0&\mbox{if } i+j-k \mbox{is an odd number}\\
\binom{k}{(j-i+k)/2}\binom{n-k}{(j+i-k)/2} &\mbox{if } i+j-k \mbox{is an even number}
\end{cases}
\end{equation}

From (\ref{estruc_cons}) is follows that
\begingroup\makeatletter\def\f@size{10}\check@mathfonts
\begin{equation}\label{producto}
\G_{n}(a)\G_{n}(b)=
\begin{cases}
\bigcup\limits_{i=0}^{a}\G_{n}(n-a-b+2i), & 0\leq a\leq \left[\dfrac{n}{2}\right], a\leq b\leq n-a,\\
\bigcup\limits_{i=0}^{n-a}\G_{n}(a+b-n+2i), & \left[\dfrac{n}{2}\right]+1\leq a\leq n, n-a\leq b\leq a.  
\end{cases}
\end{equation}
\endgroup

It follows directly from (\ref{estruc_cons}) that $\lambda_{i,j,2k+1}=0$ if $i+j$ is even and 
$\lambda_{i,j,2k}=0$ if $i+j$ is odd. The union of all sets $\G_{n}(2a)$ in $S$ will be called 
\textbf{the even partition} of $S$, and will be designated by $\E_{n}$. The \textbf{odd partition} 
$\Odd_{n}$ is defined analogously. The sets $\E_{2n}$ and $\Odd_{2n+1}$ are groups of order 
$2^{2n-1}$ and $2^{2n}$, respectively.  

The following result is taken from [9]. Let $S=\{\G_{n}(0),...,\G_{n}(n)\}$ be an $S$-set of 
$\mathbb{Z}_{2}^{n}$. A $S$-set $M$ of $S$ will be call complete maximal $S$-set if holds
\begin{enumerate}
\item $\G_{n}(a)\G_{n}(b)\supset\max S$ for all $\G_{n}(a),\G_{n}(b)$ in $S$,
\item There is no $\G_{n}(c)$ in $S$ such that $\G_{n}(c)\G_{n}(a)\supset\G_{n}(b)$ and 
$\G_{n}^{2}(c)\supset\G_{n}(a)$ for some $\G_{n}(a),\G_{n}(b)$ in $M$,
\end{enumerate}
where $\max S$ is the largest basic set of $S$.

The complete maximal $S-$sets even or odd will be designated by $\left[ \E_{n}\right] $ and 
$\left[\Odd_{n}\right]$, respectively. 

The following theorem tell us how many complete maximal $S-$sets exist in $\Z_{2}^{4n}$.

\begin{theorem}
There exist exactly two complete even maximal $S-$sets, of order $n$ and order $n+1$, in $\Z_{2}^{4n}$.
\end{theorem}

From the proof of this theorem it follows that if $\G_{4n}(a)$ is in $[\E_{4n}]$ or in $[\Odd_{4n}]$, then 
$n\leq a\leq 3n$.

\section{The Approximation}
A circulant Hadamard matrix of order $n$ is a square matrix of the form 
\begin{equation}
H =
\left(\begin{array}{cccc}
	a_{1} & a_{2} & \cdots & a_{n}\\
	a_{n} & a_{1} & \cdots & a_{n-1}\\
	\cdots & \cdots & \cdots & \cdots\\
	a_{2} & a_{3} & \cdots & a_{1}
\end{array}\right)
\end{equation}
No circulant Hadamard matrix of order larger than 4 has ever been found. This let the following
\begin{conjecture}
No circulant Hadamard matrix of order larger than 4 exists.
\end{conjecture}

Let $C$ denote the cyclic permutation on the components $+$ and $-$ of $X$ in $\Z_{2}^{n}$ such that

\begin{equation}\label{cir1}
C(X)=C\left( x_{0},x_{1},...,x_{n-2},x_{n-1}\right) =\left(x_{1},x_{2},x_{3},...,,x_{x_{0}}\right),
\end{equation}

that is, $C(x_{i})=x_{(i+1) mod n}$.
The permutation $C$ is a generator of cyclic group $\left\langle C\right\rangle$ of order $n$. 
Let $X_{C}=Orb_{\left\langle C\right\rangle }X=\{C^{i}(X):C^{i}\in
\left\langle C\right\rangle \}$. Therefore, $\left\langle C\right\rangle$
defines a partition in equivalent class on $\Z_{2}^{n}$ and this we shall denote by $\Z_{2C}^{n}$. From here it follows that all circulant Hadamard matrix $H$ is a equivalence class by $\left\langle C\right\rangle$.\\

Following Turyn and from theorem 1, if there is such $H$, then $H\in[\E_{4m^{2}}]$ or $H\in[\Odd_{4m^{2}}]$. 
Next we will show that $H$ is never contained in $[\E_{4m^{2}}]$. First we tried the following lemma

\begin{lemma}
For all $n\in\N$ and for $0\leq a\leq 2n$ it hold that
\begin{enumerate}
\item $\G_{2n}(a)=\bigcup_{i=0}^{a}\G_{n}(i)\times\G_{n}(a-i)$ if $0\leq a\leq n$.
\item $\G_{2n}(2n-a)=\bigcup_{i=0}^{a}\G_{n}(n-i)\times\G_{n}(n-a+i)$ if $n+1\leq a\leq 2n$.
\end{enumerate}
\end{lemma}
\begin{proof}
We try 1. It is clear that $\G_{n}(i)\times\G_{n}(a-i)\subset\G_{2n}(a)$, so
$\bigcup_{i=0}^{a}\G_{n}(i)\times\G_{n}(a-i)\subseteq\G_{2n}(a)$. Now pick $X$ in $\G_{2n}(a)$ and put 
$X=(B,C)$ with $B,C\in\Z_{2}^{n}$. Then $\omega(X)=\omega(B)+\omega(C)=a$ and $B\in\G_{n}(b)$, $C\in\G_{n}(c)$ with $b+c=a$. Then $\G_{2n}(a)\subseteq\bigcup_{i=0}^{a}\G_{n}(i)\times\G_{n}(a-i)$. The item 2 results from 
multiplying $\G_{2n}(a)$ by $-1$.
\end{proof}

\begin{theorem}
No circulant Hadamard matrix $X_{C}$ exists in $[\E_{4m^{2}}]$.
\end{theorem}
\begin{proof}
Pick $X$ in $\G_{4m^{2}}(2a)\in[\E_{4m^{2}}]$. From the lemma above there are $B,C$ such that 
$B\in\G_{2m^{2}}(b)$, $C\in\G_{2m^{2}}(c)$ with $b+c=2a$ and $X=(B,C)$. From (\ref{producto}) we have
\begingroup\makeatletter\def\f@size{10}\check@mathfonts
\begin{equation*}
\G_{2m^{2}}(b)\G_{2m^{2}}(c)=
\begin{cases}
\bigcup\limits_{i=0}^{b}\G_{2m^{2}}(2m^{2}-2a+2i), & 0\leq b\leq m^{2},\ b\leq c\leq 2m^{2}-b,\\
\bigcup\limits_{i=0}^{2m^{2}-b}\G_{2m^{2}}(2a-2m^{2}+2i), & m^{2}+1\leq b\leq 2m^{2},\ 2m^{2}-b\leq c\leq b.  
\end{cases}
\end{equation*}
\endgroup
then $\G_{2m^{2}}(b)\G_{2m^{2}}(c)\nsupseteq\G_{2m^{2}}(m^{2})$. Hence $\omega(XC^{2m^{2}}X)=\omega(BC,BC)=2\omega(BC)\neq2m^{2}$ so no Hadamard matrix $X_{C}$ exists in $[\E_{4m^{2}}]$.
\end{proof}

On the other hand, let $X=\{x_{i}\}$ and $Y=\{y_{i}\}$ be two complex-valued sequences of period $n$. 
The periodic correlation of $X$ and $Y$ at shift $k$ is the product defined by:

\begin{equation}
\mathsf{P}_{X,Y}(k)=\sum\limits_{i=0}^{n-1}x_{i}\overline{y}_{i+k},\ k=0,1,...,n-1,
\end{equation}

where $\overline{a}$ denotes the complex conjugation of $a$ and $i+k$ is calculated modulo $n$.
If $Y=X$, the correlation $\mathsf{P}_{X,Y}(k)$ is denoted by $\mathsf{P}_{X}(k)$ and is the autocorrelation of 
$X$. 

If $X$ is a $\Z_{2}$-sequence of length $n$, 
$\mathsf{P}_{X}(k)= 2\omega \left\{Y_{k}\right\} -n $, where $ Y_{k}=XC^{k}X $. Also by (\ref{producto}), 
if $X\in \G_{n}(a)$, then

\begin{equation}
\mathsf{P}_{X}(k)=n-4a+4i_{k},
\end{equation} 
for some $0\leq i_{k}\leq a$ and $n-\mathsf{P}_{X}(k)$ is divisible by 4 for all $k$.\\ 

Let $(\mathsf{P}_{X}(0),\mathsf{P}_{X}(1),\dots ,\mathsf{P}_{X}(n-1))$ denote the autocorrelation vector of
$X_{C}$ in $\mathbb{Z}_{2C}^{n}$. To prove the conjecture is equivalent to prove that there is no $X_{C}$ 
such that $(\mathsf{P}_{X}(0),\mathsf{P}_{X}(1),\dots ,\mathsf{P}_{X}(n-1))=(4m^{2},0,...,0)$. Thus, it is 
enough to prove that $\mathsf{P}_{X}(k)\neq 0$ for some $k\neq 0$. From theorem 2, 
$\mathsf{P}_{X}(2m^{2})\neq0$ if $X_{C}\in\G_{4m^{2}}(a)$, $a$ an even number. 
We will show that this is also true for $a$ an odd number.

\begin{theorem}
Let $\mathfrak{A}(\Z_{2C}^{n})$ denote the set of autocorrelations vector in $\Z^{n}$
and let $\theta:\Z_{2C}^{n}\rightarrow\mathfrak{A}(\Z_{2C}^{n}) $
be the mapping $\theta(X_{C})=(\mathsf{P}_{X}(0),\mathsf{P}_{X}(1),\dots ,\mathsf{P}_{X}(n-1))$. Then
$\theta$ sends the plane $\G_{n}(a)$ to the plane $X_{0}+X_{1}+\cdots +X_{n-1}=(2a-n)^{2}$.
\end{theorem}
\begin{proof}
Take $X_{C}$ in $\G_{n}(a)$. The proof follows of $\sum\limits_{k=0}^{n-1}\mathsf{P}_{X}(k)=\left(\sum\limits_{i=0}^{n-1}x_{i}\right)^{2}=(2a-n)^{2}$
\end{proof}

\begin{theorem}
No circulant Hadamard matrix $X_{C}$ exists in $[\Odd_{4m^{2}}]$, except possibly in $\G_{4m^{2}}(2m^{2}\pm m)$. \end{theorem}
\begin{proof}
From the theorem above, $\theta$ sends the plane $\G_{4m^{2}}(a)$ in the plane 
$X_{0}+X_{1}+\cdots +X_{4m^{2}-1}=(2a-4m^{2})^{2}$. As $\mathsf{P}_{X}(0)=4m^{2}$, then we have the plane 
$X_{1}+\cdots +X_{4m^{2}-1}=(2a-4m^{2})^{2}-4m^{2}$. From theorem 1 and theorem 2, $a$ take values in 
$m^{2}+2i$ with $0\leq i\leq m^{2}$. Hence $(2a-4m^{2})^{2}-4m^{2}$ is never $0$ except for $a=2m^{2}\pm m$.
Therefore there are $k$ such that $\mathsf{P}_{X}(k)\neq0$ and no circulant Hadamard matrix exists in $\G_{4m^{2}}(a)\in[\Odd_{4m^{2}}]$ with $a$ distinct to $2m^{2}\pm m$.
\end{proof}

A $(v,k,\lambda)$-difference set in a finite group $G$ of order $v$ is a $k$-subset $D$ of $G$ such that every
element $g\neq1$ of $G$ has exactly $\lambda$ representations $g=d_{1}d_{2}^{-1}$ with $d_{1},d_{2}\in D$. The order of $D$ is $k$. A circulant Hadamard matrix of order $4m^{2}$ exist if and only if there is a 
$(4m^{2},2m^{2}-m,m^{2}-m)$-difference set in the group $\mathbb{Z}_{4m^{2}}$ of order $4m^{2}$. The preceding theorem implies that if $X_{C}$ is a circulant Hadamard matrix then $X$ has Hamming weight $2m^{2}-m$ and this
agrees with the order of its difference set. 

Next, we will reduce the number of $\mathbb{Z}_{2}$-sequences in the search for circulant Hadamard matrices.

\begin{theorem}
A circulant Hadamard matrix $X_{C}$ can exist in $\G_{2m^{2}}(a)\times\G_{2m^{2}}(b)$ with $\frac{m^{2}-m}{2}\leq a\leq\frac{3m^{2}-m}{2}$ and $b=2m^{2}-m-a$.
\end{theorem}
\begin{proof}
By lemma 1, $\G_{4m^{2}}(2m^{2}-m)=\bigcup_{a+b=2m^{2}-m}\G_{2m^{2}}(a)\times\G_{2m^{2}}(b)$ and by (\ref{producto})
\begingroup\makeatletter\def\f@size{11}\check@mathfonts
\begin{equation*}
\G_{2m^{2}}(a)\G_{2m^{2}}(b)=
\begin{cases}
\bigcup\limits_{i=0}^{a}\G_{2m^{2}}(m+2i), & 0\leq a\leq m^{2},\ a\leq b\leq 2m^{2}-a,\\
\bigcup\limits_{i=0}^{2m^{2}-a}\G_{2m^{2}}(-m+2i), & m^{2}+1\leq a\leq 2m^{2},\ 2m^{2}-a\leq b\leq a.  
\end{cases}
\end{equation*}
\endgroup
As $\G_{2m^{2}}(a)\G_{2m^{2}}(b)$ must contain $\G_{2m^{2}}(m^{2})$ then $\frac{m^{2}-m}{2}\leq a\leq m^{2}$ and from $a\leq b\leq 2m^{2}-a$ and $b=2m^{2}-m-a$ it follows that $a\leq\frac{2m^{2}-m-1}{2}$. Hence 
$\frac{m^{2}-m}{2}\leq a\leq\frac{2m^{2}-m-1}{2}$. Equally, $m^{2}+1\leq a\leq\frac{3m^{2}-m}{2}$ and from
$2m^{2}-a\leq b\leq a$ it follows that $\frac{2m^{2}-m+1}{2}\leq a\leq\frac{3m^{2}-m}{2}$. Thereby
$\frac{m^{2}-m}{2}\leq a\leq\frac{3m^{2}-m}{2}$.
\end{proof}

\begin{theorem}
A circulant Hadamard matrix $X_{C}$ can exist in $\G_{2m^{2}}(a)\times\G_{2m^{2}}(b)$ with $\frac{m^{2}+m}{2}\leq 2m^{2}-a\leq\frac{3m^{2}+m}{2}$ and $b=m+a$.
\end{theorem}
\begin{proof}
Multiplying $\G_{2m^{2}}(a)\times\G_{2m^{2}}(b)$ by $-1$ in the previous theorem.
\end{proof}

From the previous theorems the search for circulant Hadamard matrices in a set of order
$2\binom{4m^{2}}{2m^{2}-m}$ is reduced to the search in a set of order
\begin{equation*}
2\sum_{a=(m^{2}-m)/2}^{(3m^{2}-m)/2}\binom{2m^{2}}{a}\binom{2m^{2}}{2m^{2}-m-a}
\end{equation*}

\Addresses

\end{document}